\documentclass{article}
\usepackage[utf8]{inputenc}

\usepackage{amssymb,amsmath,amsthm,color}
\usepackage{faktor}
\usepackage{hyperref}
\usepackage{cleveref}
\usepackage{csquotes}
\usepackage{graphicx}
\usepackage{geometry}

\usepackage{thmtools}
\usepackage{thm-restate}
\usepackage{bbm}
\usepackage{dsfont}

\newcommand{\Z}{\mathbb{Z}}

\newtheorem{theorem}{Theorem}
\newtheorem*{theorem*}{Theorem}
\newtheorem{lemma}[theorem]{Lemma}
\newtheorem{observation}[theorem]{Observation}
\newtheorem{remark}[theorem]{Remark}

\newtheorem{question}[theorem]{Question}
\newtheorem{conjecture}[theorem]{Conjecture}

\title{The number of tiles of $\Z^d$}
\date{\today}

\author{Itai Benjamini, Gady Kozma, Elad Tzalik }
\begin{document}

\maketitle
\begin{abstract}
    \noindent It is proved that the number of subsets of $[n]^d$ that tile $\Z^d$ is $\left(3^{\frac{1}{3}}\right)^{n^d \pm o(n^d)}$.
\end{abstract}

\section{Introduction}

A set \(S\subseteq \Z^d\) tiles \(\Z^d\) by translations if there is a set \(T\subseteq\Z^d\) such that
\[
    \Z^d=\dot\bigcup_{z\in T}(S+z).
\]
Tiling $\Z^d$
with translations is a natural problem that has been widely studied \cite{COVEN,Wang, Newman,LabaLondner,LabaLondner2,GT}. Even in dimension $1$ the exact structure of tiles remains unclear in general, with a long-standing conjecture of Coven and Meyerowitz \cite{COVEN} aiming to shed light on this problem.

Rather than studying the \emph{exact} structure of tiles, which is open even in dimension \(1\), our aim is to study their \emph{typical} structure by counting them. Concretely, we count how many subsets of $[n]^d = \{(x_1,\ldots,x_d) \mid x_i \in \{1,\ldots,n\} \}$ can tile $\Z^d$, which implies bounds on the search space we are working with. 
We remark that our work differs from the recent work of Stern \cite{Stern}, which studied the problem of counting for a given $S \subseteq \Z^d$, how many different ways there are to tile $S$ ($S$ may not tile $\Z^d$).

Let $t_{n,d}$ be the number of subsets of $[n]^d$ that tile $\Z^d$. We use the notation $v_n:=n^d$ for the size of $[n]^d$. We prove:

\begin{theorem}\label{thm:number-of-tiles}
For every fixed \(d\ge 1\),
\[
    (3^{1/3})^{v_n-o(v_n)}
    \le t_{n,d}\le
    (3^{1/3})^{v_n+o(v_n)} .
\]
\end{theorem}

\section{Proof of \Cref{thm:number-of-tiles}}

\subsection{Lower Bound}
Assume first that \(n=3t\). In dimension \(1\), choose one point from each residue class modulo \(t\) inside \([3t]\). There are \(3^t=(3^{1/3})^n\) choices, and each such set tiles \(\Z\) by translations by \(t\).

For general \(d\), apply this construction independently on each line parallel to \(e_1\). Equivalently, for every
\[
    (i,x_2,\ldots,x_d)\in [t]\times[3t]^{d-1},
\]
choose exactly one of the three points
\[
    (i,x_2,\ldots,x_d),\quad
    (i+t,x_2,\ldots,x_d),\quad
    (i+2t,x_2,\ldots,x_d).
\]
The resulting set is a fundamental domain for the lattice generated by
\[
    t e_1,\quad 3t e_2,\ldots,3t e_d,
\]
and hence tiles \(\Z^d\). This gives
\[
    t_{3t,d}\ge 3^{t(3t)^{d-1}}
    =
    (3^{1/3})^{(3t)^d}.
\]
For general \(n\ge 3\), let \(n'=3\lfloor n/3\rfloor\). Applying the construction inside
\([n']^d\subseteq[n]^d\) gives
\[
    t_{n,d}\ge (3^{1/3})^{(n')^d}
    =
    (3^{1/3})^{n^d-o(n^d)}.
\]

\subsection{Upper bound}
We now prove the upper bound. Since \(H(0.1)<\frac13\log_2 3\),
\[
    \sum_{r\le 0.1v_n}\binom{v_n}{r}
    \le 2^{H(0.1)v_n}
    =
    o\!\left((3^{1/3})^{v_n}\right).
\]
Thus it remains to count tiles \(S\subseteq[n]^d\) with
\[
    |S|\ge \alpha v_n,\qquad \alpha=0.1.
\]

Let \(k=k(n)\) be a positive integer with $k\to\infty$; at the end we take $k=\lfloor n^{1/(2d)}\rfloor$. Consider shifts $z_1,\ldots,z_\ell\in\Z^d$. We say that $S$ tiles $[kn]^d$ with $z_1,\ldots,z_\ell$ if the translates $S+z_a$ are pairwise disjoint and
\[
    [kn]^d\subseteq \dot\bigcup_{a=1}^{\ell}(S+z_a).
\]
Call the translates efficient if \(([n]^d+z_a)\cap[kn]^d\ne\emptyset\) for every \(a\). Crucially, every set $S$ that tiles $\mathbb{Z}^d$ tiles $[kn]^d$ with a choice of efficient translates. By definition, every efficient shift lies in \(\{-n,\ldots,kn\}^d\). Also, \(\ell\ge k^d\), because \(|S|\le v_n\) and the translates cover \([kn]^d\), while
\[
    \ell\le \frac{(k+2)^d}{\alpha}=O_\alpha(k^d),
\]
because the efficient translates are disjoint, each has size at least \(\alpha v_n\), and all lie in a slightly enlarged cube of volume \(O(k^d v_n)\). Therefore the number of efficient sequences is at most
\[
    (O(k^d v_n))^{O_\alpha(k^d)}.
\]
We fix $z_1,\ldots,z_{\ell}$ and count the number of sets $S$, that tile a set which contains $[kn]^d$ with the fixed translations $z_1,\ldots,z_{\ell}$. Define
\[
    A_y=\{\,y-z_a:a\in[\ell]\,\}\cap[n]^d,
    \qquad y\in[kn]^d.
\]
\noindent Since the translates \(S+z_a\) are pairwise disjoint and cover \([kn]^d\), for every \(y\in[kn]^d\) the set
\[
    F_y=S\cap A_y
\]
has size exactly one.

Let \(\mathcal S\) be chosen uniformly from the family of all sets \(S\subseteq[n]^d\) with \(|S|\ge\alpha v_n\) such that the fixed translates \(S+z_1,\ldots,S+z_\ell\) are pairwise disjoint and cover \([kn]^d\). We bound the entropy of \(\mathcal S\) and show \(H(\mathcal S) \leq \log_2 \left( 3^{\frac{1}{3}}+ o_{k}(1) \right) \cdot v_n \), which implies that for every fixed \(z_1,\ldots,z_\ell\) we have at most \(\left(3^{\frac{1}{3}}+o_{k}(1)\right)^{v_n}\) tiles. Union bounding over the efficient sequences counted above will conclude the proof \footnote{For $k=\lfloor n^{1/(2d)}\rfloor$, the term coming from this union bound is $(3^{\frac{1}{3}})^{o(v_n)}$.}.

Write
\[
    \mathcal S=(X_x)_{x\in[n]^d},
\]
where \(X_x=\mathbbm{1}_{x\in\mathcal S}\). For each \(x\in[n]^d\), the set of \(y\in[kn]^d\) such that \(x\in A_y\) is
\[
    \{x+z_a:a\in[\ell]\}\cap[kn]^d.
\]
\noindent
For \(x\in[n]^d\), set
\[
    J(x)=|\{a\in[\ell]:x+z_a\notin[kn]^d\}|.
\]
If \(x+z_a\notin[kn]^d\), then, by efficiency, \([n]^d+z_a\) crosses the boundary of \([kn]^d\), and hence \(S+z_a\) lies in the \(n\)-neighbourhood \(B\) of the boundary of \([kn]^d\). Since the corresponding translates are disjoint,

\begin{align}\label{eq - vol packing}
    \alpha v_n J(x)\le |B|
       \le ((k+2)^d-(k-2)^d)v_n.
\end{align}
Thus, for all sufficiently large \(k\),
\[
    J(x)\le J:=\frac{5dk^{d-1}}{\alpha}.
\]
Therefore the family \((A_y)_{y\in[kn]^d}\) covers each coordinate \(x\in[n]^d\) at least \(\ell-J\) times.

To bound \(H(\mathcal S)\), we use Shearer's inequality, which gives an upper bound on the entropy of \(\mathcal S\) in terms of entropies of its projections. Concretely:

\[ H(\mathcal S)\leq \frac{1}{\ell-J} \sum_{y\in[kn]^d} H(\mathcal S|_{A_y}) .\]

\noindent Here comes the crux of the proof - since $\mathcal{S}$ tiles $[kn]^d$ using the translations, for every \(y\in[kn]^d\), the random vector \(\mathcal S|_{A_y}\) has exactly one coordinate equal to \(1\). If no coordinate of \(\mathcal S|_{A_y}\) is $1$ then $y$ is not covered by the translates of $\mathcal{S}$. If two coordinates of \(\mathcal S|_{A_y}\) are $1$, then $y$ is covered by two distinct elements of $\mathcal{S}$, violating the tiling assumption. Consequently, the random variable \(\mathcal S|_{A_y}\) can take at most \(|A_y|\) possible values (which is much better than the naive $2^{|A_y|}$ bound, and is the reason for the improved size bound for tiles), hence
\[
    H(\mathcal S|_{A_y})\le \log_2 |A_y|.
\]
Thus, applying Shearer's inequality we get:

\begin{align}\label{eq : entropy A set bound}
    H(\mathcal S)
    \le \frac{1}{\ell-J}\sum_{y\in[kn]^d}\log_2|A_y|
    \le \frac{1}{\ell-J}\log_2\prod_{y\in[kn]^d}|A_y|.
\end{align}

\noindent We will need the following observation (which is a discrete form of Jensen's inequality):

\begin{observation}\label{obs : discrete jensen}
    Let $x_1,\ldots ,x_m$ be non-negative integers with $\sum_{i=1}^m x_i = t$, with $t=am+b$ and \(0\le b<m\). Then
    \[
        \prod_{i=1}^m x_i\le (a+1)^b a^{m-b}.
    \]
\end{observation}

\begin{proof}
    Since for $k\geq 0$ and $s>1$ we have $k(k+s)<(k+\lfloor \frac{s}{2} \rfloor) (k+\lceil \frac{s}{2} \rceil)$ any maximum contains elements $x_i$ that differ by at most $1$ from one another. The only choice (up to permutation) of $x_1,\ldots ,x_m$ with this property achieves the desired upper bound.
\end{proof}

\noindent Set \(x_y=|A_y|\), \(y\in[kn]^d\). Then
\[
    \sum_{y\in[kn]^d}|A_y|\le \ell v_n,
\]
because each \(x\in[n]^d\) contributes to at most one \(A_y\) for each shift \(z_a\). The product is maximized, under this upper bound on the total sum, by using the full allowed sum. Applying \Cref{obs : discrete jensen} with \(m=(kn)^d\) and total sum at most \(\ell v_n\), and using $\ell n^d = \lfloor \frac{\ell}{k^d} \rfloor \cdot (kn)^d + \left\{ \frac{\ell}{k^d} \right\} \cdot (kn)^d$ where $\{\}$ denotes the fractional part of a real number, we get:

\begin{align}
    H(\mathcal S) &\leq \frac{1}{\ell-J} \log_2 \left( \left(\Bigg\lfloor \frac{\ell}{k^d} \Bigg\rfloor+1 \right)^{\left\{ \frac{\ell}{k^d} \right\} \cdot k^d v_n} \Bigg\lfloor \frac{\ell}{k^d} \Bigg\rfloor^{(1-\left\{ \frac{\ell}{k^d} \right\})k^d v_n} \right) \\
    &= \frac{1}{\ell-J} \log_2 \left( \left(\Bigg\lfloor \frac{\ell}{k^d} \Bigg\rfloor+1 \right)^{\left\{ \frac{\ell}{k^d} \right\} \cdot k^d} \Bigg\lfloor \frac{\ell}{k^d} \Bigg\rfloor^{(1-\left\{ \frac{\ell}{k^d} \right\})k^d} \right)\cdot v_n \\
    &= \frac{\ell}{\ell-J} \log_2 \left( \left(\Bigg\lfloor \frac{\ell}{k^d} \Bigg\rfloor+1 \right)^{\left\{ \frac{\ell}{k^d} \right\} \cdot \frac{k^d}{\ell}} \Bigg\lfloor \frac{\ell}{k^d} \Bigg\rfloor^{(1-\left\{ \frac{\ell}{k^d} \right\}) \cdot \frac{k^d}{\ell}} \right)\cdot v_n \label{eq : fractional quantity}
\end{align}

To bound the term inside the $\log$ by $3^{\frac{1}{3}}$, we use the following lemma:

\begin{lemma}
    Let $N$ be a positive integer and $r \in [0,1]$. Then
    \[
        f(r)=(N+1)^{\frac{r}{N+r}}N^{\frac{1-r}{N+r}}
    \]
    attains its maximum on \([0,1]\) at an endpoint.
\end{lemma}

\begin{proof}
    Let
    \[
        g(r)=\log f(r)
        =\frac{r\log(N+1)+(1-r)\log N}{N+r}.
    \]
    Then
    \[
        g'(r)=\frac{N\log(1+1/N)-\log N}{(N+r)^2},
    \]
    whose sign is independent of \(r\). Hence \(g\), and therefore \(f\), is monotone on \([0,1]\).
\end{proof}

    By the lemma, the term inside the logarithm in \Cref{eq : fractional quantity} is bounded by
    \(M^{1/M}\) for some positive integer \(M\). Notice that since $x^{\frac{1}{x}}$ increases at $[1,e)$ and decreases at $(e,\infty)$, together with the fact that $2^{\frac{1}{2}}< 3^{\frac{1}{3}}$ we know that:
    
    \[
        \max_{M\in\mathbb N} M^{1/M}=3^{1/3}.
    \]

    \;
    
    \noindent therefore $H(\mathcal S) \leq \frac{\ell}{\ell-J} \log_2(3^{\frac{1}{3}}) \cdot v_n$, hence the number of possible sets \(S\) for these fixed shifts and $|S| \geq \alpha v_n$ is at most:
    
    \[\left(3^{\frac{1}{3}} \right)^{\frac{\ell}{\ell-5dk^{d-1}/\alpha}v_n}\]

\;

\noindent In conclusion, setting $k=\lfloor n^{\frac{1}{2d}}\rfloor$ concludes the proof of the upper bound.

Since \(\ell\ge k^d\) and \(J\le 5dk^{d-1}/\alpha\),
\[
    \frac{\ell}{\ell-J}
    \le
    \frac{k^d}{k^d-5dk^{d-1}/\alpha}
    =
    1+O_\alpha(1/k).
\]

\[
t_{n,d}
\le
\sum_{r\le \alpha v_n}\binom{v_n}{r}
+
(O(k^d v_n))^{O_\alpha(k^d)}
\cdot
\left(3^{1/3}\right)^{
    \frac{k^d}{k^d-5dk^{d-1}/\alpha}v_n
}
\le
(3^{1/3})^{v_n+o(v_n)}.
\]
With \(k=\lfloor n^{1/(2d)}\rfloor\), the sequence-counting factor is \((3^{1/3})^{o(v_n)}\), and the exponent factor is \(1+o(1)\).

\begin{remark}
    It also follows from the proof that all translation sequences for which $\frac{\ell}{k^d} \notin (3-\varepsilon, 3+\varepsilon)$ yield a stronger bound on the entropy. This implies that as $n$ goes to $\infty$ the size of a uniform random tile divided by $n^d$ converges to $\frac{1}{3}$ a.a.s. \end{remark}
\;

\section{Concluding remarks}

The upper-bound argument uses only two coarse facts about boxes in \(\Z^d\): efficient translates lie in a slightly enlarged region, and the \(n\)-neighbourhood of the boundary of \([kn]^d\) has \(o(k^d n^d)\) points when \(k\to\infty\). Thus the same proof applies, with minor changes, to other Følner-type regions in \(\Z^d\), such as lattice balls for a fixed norm. In spaces of exponential growth this boundary layer may have volume comparable to the whole ball, and the constant \(3^{1/3}\) does not seem to be the right one.

We conclude with several conjectures and questions:

\begin{conjecture}
    Fix \(r\ge 1\). Let \(x\in[n]^d\) be chosen uniformly, and let \(\mathcal S_n\) be a uniformly random tile contained in \([n]^d\), independently of \(x\). Does the law of
    \[
        \bigl(\mathbbm{1}_{x+u\in\mathcal S_n}\bigr)_{u\in B_r(0)}
    \]
    converge, as \(n\to\infty\), to the product measure of Bernoulli-\(1/3\) random variables?
\end{conjecture}

\begin{conjecture}
    For each $d$ there is a constant $C(d)$ that depends only on $d$ such that: $t_{n,d} = O(n^{C(d)} \cdot (3^{\frac{1}{3}})^{v_n})$. Can one take $C(d)=d-1$?
\end{conjecture}

\begin{question}
    Let \(\mathcal S\) be a tile contained in $[n]^d$ chosen uniformly at random. What is the size of the largest connected component in the subgraph induced by \(\mathcal S\)? The lower-bound family resembles site percolation on \([n]^{d-1}\times[n/3]\) with \(p=1/3\). Since site percolation on \(\Z^d\) with \(p=1/3\) is supercritical for \(d\ge 3\) and subcritical for \(d\le 2\)\footnote{See \url{https://en.wikipedia.org/wiki/Percolation_threshold}}, this suggests asking whether a similar phase transition occurs for the largest connected component of a random tile.
\end{question}

\begin{question}
    Let $M\subseteq \Z$ be any subset of size $n$. Are there at most $\left( 3^{1/3} + o(1) \right)^n$ tiles inside $M$? Is the number of tiles contained in \(M\) maximized when \(M\) consists of \(n\) consecutive integers?
\end{question}

\begin{question}
    How many tiles are contained in the set \(\{1,2,4,\ldots,2^n\}\)? Are the only such tiles the sets of size \(1\) and \(2\)? \footnote{Notice that if one considers \(0,1,2,4,\ldots,2^n\), then for values of \(n\) admitting a prime \(p\approx n/3\) for which \(2\) generates \(\mathbb F_p^\times\), one can construct many tiles by picking one from each mod \(p\) class to be in the tile.}
\end{question}

\begin{question}
    Call $S \subseteq \Z^d$ \emph{perfect} if all non-empty subsets of $S$ tile $\Z^d$. E.g. it follows immediately by the work of \cite{Newman} that the largest size of a perfect subset of \(\Z\) is \(3\). What is the size of the largest perfect subset in $\Z^d$ for $d \geq 2$?
\end{question}

\paragraph{Acknowledgment.} I.B. and G.K. thank the Israel Science Foundation for its support. We thank Tom Meyerovitch for his comments on an earlier version of this work. We thank the referee for suggestions that improved the paper.

\bibliographystyle{plainurl}
\bibliography{ref}

\end{document}